\documentclass{amsart}[12pt]

\usepackage{geometry}
\usepackage{amsmath, amssymb, amsthm, amsfonts, amsxtra, mathrsfs, mathtools, dsfont} 

\usepackage{tikz-cd}
\usepackage{hyperref}

\newtheorem{innercustomgeneric}{\customgenericname}
\providecommand{\customgenericname}{}
\newcommand{\newcustomtheorem}[2]{%
  \newenvironment{#1}[1]
  {%
   \renewcommand\customgenericname{#2}%
   \renewcommand\theinnercustomgeneric{##1}%
   \innercustomgeneric
  }
  {\endinnercustomgeneric}
}

\newcustomtheorem{customthm}{Theorem}
\newcustomtheorem{customlemma}{Lemma}
\newcustomtheorem{customcor}{Corollary}

\newtheorem{theorem}{Theorem}[section]
\newtheorem{lemma}[theorem]{Lemma}

\newtheorem{corollary}[theorem]{Corollary}

\theoremstyle{definition}

\theoremstyle{remark}

\def\R{\mathbb{R}}

\def\calO{\mathcal{O}}
\def\calP{\mathcal{P}}
\def\calV{\mathcal{V}}

\usepackage{hyperref}
\usepackage{float}

\title{
    Universality of Polyhedral Linkages
}

\author{
    Robert Miranda
}

\begin{document}

\begin{abstract}
    Planar linkages are a rich area of study motivated by practical applications in engineering mechanisms. A central result is Kempe's Universality Theorem, which states that semi-algebraic sets can be realized by planar linkages. Polyhedral linkages are generalizations of planar linkages to higher dimensions, where the faces are required to be rigid. In this paper, we generalize Kempe's Universality Theorem to polyhedral linkages with an embedded construction in dimension three and above.
\end{abstract}

\maketitle{}

\section{Introduction}
\label{section:introduction}

\subsection{Setup}

A planar linkage is an edge weighted graph that is realized in $\R^2$, such that the weight of each edge corresponds to the distance between the two adjacent vertices. One may also fix some vertices to be realized at a specific point in $\R^2$, thus constraining the possible realizations of the other vertices. Planar linkages can be used to `compute' functions in the following sense. Suppose $U \subset \R^2$ is an open set and $F : U \to \R^2$ is a function. We say that a planar linkage \textit{defines} $F$ on the neighborhood $U$ if the linkage has an input vertex $x$ and an output vertex $y$, such that whenever $x$ is realized in $U$, then $y$ is constrained to be realized at $F(x)$. (See e.g. \cite{DO'R}, \cite{Pak}.)

\smallskip

Famously, planar linkages are universal in the sense that they can realize any polynomial function in any bounded region. This result is known as Kempe's Universality Theorem \cite{Kem}. (See a generalization in modern terminology in \cite{KM}.) Alternatively, we may fix the output vertex to a specific point, which constrains the input vertex to trace the zero locus of the defined function. Such linkages are called \textit{closed}. In this setting, planar linkages can trace any bounded region of an algebraic set, or as popularized by Thurston, they can ``sign your name''. (See ~\cite{King99} for a historical account.)

\smallskip

We say that a planar linkage is \textit{embedded} if its edges do not cross and its vertices are distinct. Kempe's original proof did not consider embedding, and critically relied on linkages which are not planar graphs. Nevertheless, Kempe's Universality Theorem was proved with an embedded construction in \cite{AB+}, settling a question of Shimamoto from 2004.

\subsection{Main Results}

In this paper, we prove a version of Kempe's Universality Theorem that holds for embedded linkages in dimension three and generalizes to higher dimensions. We consider \textit{polyhedral linkages}, which replace the weighted graphs used in planar linkages with polyhedral complexes of codimension $1$ satisfying certain properties. The term polyhedral linkage was first coined by Goldberg in \cite{Gol}. In three dimensions, polyhedral linkages are closely related to notions of \textit{rigid origami}, but in this case our surface may have any topology. (See \cite{DO'R}, \cite{GZ}.) Our definition is modeled after the definition of planar linkages given by Kapovich and Millson in \cite{KM}. (See \S~\ref{section:polyhedral-linkages}.) We note versions of Kempe's Universality Theorem have previously been proven for weighted graphs embedded in $\R^n$ by \cite{King98} and \cite{Abb}.

\begin{theorem}
    \label{thm:linkage-universal}
    Let $U \subset \R^3$ be a bounded, open set, and $F : U \to \R^n$ be a polynomial function. There exists an embedded, functional polyhedral linkage $\calP$ which defines $F$ on $U$. Moreover, this construction generalizes to all higher dimensions.
\end{theorem}

We also generalize the corollary of Kempe's Universality Theorem for closed linkages. That is, to use Thurston's phrasing, there exist embedded polyhedral linkages which can ``sign'' any algebraic set in $\R^n$.

\begin{corollary}
    \label{thm:linkage-universal-closed}
    Let $U \subset \R^3$ be a bounded, open set, and $S \subset \R^3$ be an algebraic set. Then there is an embedded polyhedral linkage which realizes $S \cap U$. Moreover, this construction generalizes to all higher dimensions.
\end{corollary}

\subsection*{Remark: (Embedded linkages with multiple inputs and outputs)}
Kempe's Universality Theorem extends to a stronger result that for constants $m_1, m_2 \geq 1$, each polynomial function $F : (\R^2)^{m_1} \to (\R^2)^{m_2}$ can be realized by planar linkages on an open bounded set. (See \cite{KM}.) However, allowing multiple inputs and outputs poses a separate problem for embedded linkages. For example, consider a function with two inputs. If this were simulated by a physical machine, then the two input vertices may be realized at the same point. This would mean that the physical mechanisms should be able to `pass through each other'. A similar problem can occur with output vertices.

\smallskip

To address this problem, we propose a modification to the definition of a functional linkage which considers each input and output vertex relative to its own reference coordinate frame. Specifically, in the forgetful maps defining the input and output vertices, we introduce a postcomposition which translates each coordinate. (See \S\ref{section:polyhedral-linkages}, and \cite{KM} for comparison.) This allows us to extend the vector version of Kempe's Universality Theorem to three dimensions and above with embedded linkages.

\begin{figure}
\begin{center}
    \includegraphics[scale = 0.5]{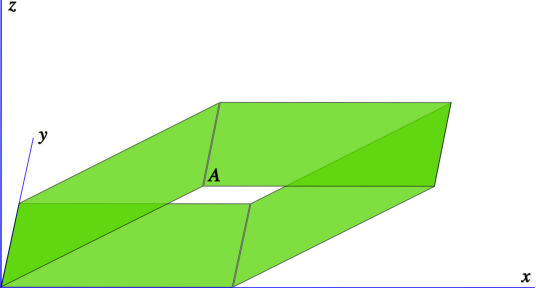}
    \caption{A polyhedral linkage.}
    \label{fig:polyhedral-linkage}
\end{center}
\end{figure}

\subsection{Outline of the construction}
Our construction of functional polyhedral linkages centers around an efficient way to generate linear motion in three dimensions. In $\R^2$, finding a linkage which produces linear motion was a fundamental problem in the field. (For a history of the solution, see \cite[\S 14]{KM}.) However, the construction in $\R^3$ is comparatively simpler. In fact, in our current terminology a polyhedral linkage which achieves linear motion was first discovered by Sarrus in 1853, several years before any solution in $\R^2$ has been published. (See \cite{Sar}, \cite{Gol} and also \cite{WKA}.) Our construction is based on Sarrus' linkage with modifications to create a periodic flexing. Specifically, for any $0 < a < b$, we produce an embedded polyhedral linkage whose length in one dimension varies continuously in the interval $(a, b)$, and which is arbitrarily small in all other dimensions. (See \S~\ref{section:linear-motion}.)

\smallskip

We set up an array of these \textit{extender} linkages arranged parallel to each other in a fixed plane. The length of each extender encodes a scalar value, and the register of extenders allows us to encode a vector value. Note that this differs fundamentally from registers used in digital computers because each extender stores a value with infinite precision, but the range of values an extender can hold is still limited by its construction, i.e., by the chosen values of $a$ and $b$. (See \S~\ref{section:scalar-computation}.)

\smallskip

Computations are decomposed into a sequence of elementary operations. Each elementary operation uses an input register and an output register, which is a parallel translate of the input register. The output of the operation is stored in an extender in the output register by attaching both the output extender, and all input extenders, to a specific elementary linkage. (See \S~\ref{section:scalar-computation}) These are generalizations of elementary planar linkages to higher dimensions, sometimes with slight modifications so that they are embedded. (See \S \ref{section:polyhedral-linkages} and \cite[\S 6]{KM} for comparison.) All other values in the original register are faithfully copied to the output register by rigid linkages.

\subsection{Structure of the paper}
We prove our main results, Theorem~\ref{thm:linkage-universal} and Corollary~\ref{thm:linkage-universal-closed}, by showing that all polynomial functions can be decomposed into elementary operations which can be performed by polyhedral linkages. Then we construct a polyhedral linkage in $\R^3$ which has a point that can be moved freely in a $3$-dimensional region, and which records the coordinates of this point in a register of $3$ extenders. (See \S~\ref{section:vector-computation}.) We use this register as the beginning of our computation, and attach a similar linkage to combine the final output register into the coordinates of a single output vertex. The generalization to multiple inputs and outputs is simple with our convention of separate reference frames. (See \S~\ref{section:polyhedral-linkages}.) The general case in higher dimensions is then immediate by extruding our construction. (See \S~\ref{sub-section:higher-dim}.)

\section{Polyhedral linkages}
\label{section:polyhedral-linkages}

Here we give a general definition of a polyhedral linkage in dimensions two or higher. All of the constructions take place in two or three dimensions, and generalizations to higher dimensions are immediate. Our definition is a direct generalization of the definition of a planar linkage to include a codimension $1$ polyhedral complex realized in $\R^{n}$. See \cite{KM} for motivation and a full definition for planar linkages.

\subsection{Polyhedral linkages}
\label{sub-section:polyhedral-linkages}
Let $\calP$ be a polyhedral complex in $\R^n$. We say that $\calP$ is \textit{pure} if every maximal polytope has the same dimension, and we say that $\calP$ is \textit{proper} if the intersection of any two $k$-dimensional faces in $\calP$ is either $\varnothing$, or a $(k-1)$-dimensional face in $\calP$. An edge weighted graph is exactly a $1$-dimensional pure, proper polyhedral complex in $\R^2$.

\smallskip

A \textit{polyhedral linkage} of dimension $n$ is a pair $(\calP, W)$, where $\calP$ is a $n$-dimensional pure, proper polyhedral complex in $\R^{n+1}$, and $W \subset \calV(\calP)$ is a subset of the vertices of $\calP$ called the \textit{fixed vertices}. When clear by context, we will refer to the polyhedral linkage as $\calP$ without reference to $W$.

\smallskip

A \textit{realization} of $\calP$ is a map $\phi : \calV(\calP) \to \R^{n+1}$ such that for each maximal face $F \in \calP$, the vertices $\{\phi(v) \ | \ v \in F \cap \calV(\calP)\}$ form the vertices of a polytope congruent to $F$. I.e., the maximal polytopes can be rearranged by individual rigid motions as long as they all fit together in the same structure. The $(n-1)$-dimensional faces act as ``joints'', allowing the structure to flex or hinge. The set of all realization, $C(\calP)$, is called the \textit{configuration space} of $\calP$.

\smallskip

Given a set of points $Z$ in $\R^{n+1}$ in bijection with $W$, a realization $\phi : \calV(\calP) \to \R^{n+1}$ is said to be \textit{relative to} $Z$ if $\phi(w) = z$ for each $w \in W$ and the corresponding $z \in Z$. The set of all relative realizations to $Z$, $C(\calP, Z)$, is called the \textit{relative configuration space}.

\smallskip

We say that a realization $\phi \in C(\calP, Z)$ is \textit{embedded} if the interiors of all maximal faces are pairwise disjoint. The set of all embedded realizations, $C^e(\calP, Z)$, forms an open subset of $C(\calP, Z)$. 

For example, consider the $2$-dimensional polyhedral linkage $\calP$ formed by removing two opposite faces of a cube. Let $W$ be the vertices of one face of the cube, and let $Z$ be the four points $\{(0,0,0), (1,0,0), \newline (1,1,0), (0,1,0)\}$ in the $xy$-plane. (See Figure \ref{fig:polyhedral-linkage}.) The relative configuration space $C(\calP, Z)$ consists of three intersecting smooth curves and is naturally identified with the moduli space of the square, \cite[~\S~3]{KM}. However, the embedded realization space $C^e(\calP, Z)$ is identified with the set $\{x^2 + y^2 =~1 ; z~\neq~0\}$, where the location of the vertex $A$ determines the entire configuration when $A$ is not on the $xy$-plane. The points where $A$ is on the $xy$-plane correspond to self-intersecting realizations of $\calP$, and are included in the other two curves in $C(\calP, Z)$. Any physical model which is built to emulate the behavior of $\calP$ can naturally move within a single connected component of $C^e(\calP, Z)$.

\subsection{Functional linkages}
Next we define \textit{functional linkages}. Let $m_1, m_2 \geq 1$ and let 

\[
    F : \R^{(n+1)m_1} \longrightarrow \R^{(n+1)m_2}
\] 
be a given function. A functional linkage is a polyhedral linkage $\calP$ of dimension $n$ with two distinguished sets of vertices $\{P_1, \dotsc, P_{m_1}\}$ called \textit{input vertices} and $\{Q_1, \dotsc, Q_{m_2}\}$ called \textit{output vertices}. We also define two forgetful maps $p : C(\calP, Z) \to \R^{(n+1)m_1}$ and $q : C(\calP, Z) \to \R^{(n+1)m_2}$ as
\[
    p(\phi) = (\phi(P_1) + X_1, \dotsc, \phi(P_{m_1}) + X_{m_1}),
\]
\[
    q(\phi) = (\phi(Q_1) + Y_1, \dotsc, \phi(Q_{m_2}) + Y_{m_2}),
\]
for some choice of translations $X_1, \dotsc, X_{m_1}, Y_1, \dotsc, Y_{m_2} \in \R^{n+1}$. We say that $\calP$ \textit{defines} the function $F$ at a point $\calO \in \R^{(n+1)m_1}$ if there is a commutative diagram
\[\begin{tikzcd}
    & {C(\calP, Z)} \\
    {\R^{(n+1)m_1}} && {\R^{(n+1)m_2}}
    \arrow["p"', from=1-2, to=2-1]
    \arrow["q", from=1-2, to=2-3]
    \arrow["F", from=2-1, to=2-3]
\end{tikzcd}\]
and $p$ is a regular topological branched cover of a bounded open set $U \subset \R^{(n+1)m_1}$ containing the point $\calO$. Alternatively, we say that $\calP$ is a \textit{functional linkage} for the germ $(F, \calO)$. Moreover, we say that a functional linkage $\calP$ is \textit{embedded} if there exists a connected component $E \subset C^e(\calP, Z)$ such that the restrictions $p|_E, q|_E$ still form a commutative diagram. 

\[\begin{tikzcd}
    & E \\
    {\R^{(n+1)m_1}} && {\R^{(n+1)m_2}}
    \arrow["{p|_E}"', from=1-2, to=2-1]
    \arrow["{q|_E}", from=1-2, to=2-3]
    \arrow["F", from=2-1, to=2-3]
\end{tikzcd}\]

Finally, we say that a functional linkage $\calP$ is \textit{closed} if the output vertices are also fixed vertices. In this case, the image of the input map corresponds to the zero set of the defined function, and we say the $\calP$ \textit{realizes} $U$, where $U := p(C(\calP, Z))$, or $U := p|_E(E)$ in the embedded case.

\section{Linear motion in three dimensions}
\label{section:linear-motion}

In this section we give a construction which is a modification of Sarrus' linkage for linear motion. It has the added property that the linkage is periodic, and remains periodic during flexion. This is important for scalability in the construction of embedded polyhedral linkages.

\subsection{Basic construction}
Consider the following planar linkage. The vertices $A$ and $B$ are fixed at $(0,0)$ and $(1,0)$, respectively. All other vertices are allowed to move freely, and all edges have length $1$. See Figure~\ref{fig:planar-linkage}.

\smallskip

\begin{figure}
\begin{center}
    \includegraphics[scale = 1]{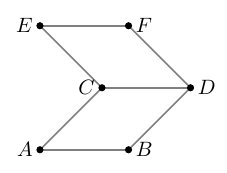}
    \caption{A planar linkage which achieves a $2$-dimensional range of motion.}
    \label{fig:planar-linkage}
\end{center}
\end{figure}

The vertex $E$ may move within the $2$-dimensional region $\{x^2 + y^2 \leq 4\}$. However, when this planar linkage is extruded\footnote{To \textit{extrude} means to press or push out. In this context, we use the word extrude to mean forming a new polyhedral complex by taking the product with the unit interval. E.g., a line segment is extruded to a square, which is extruded to a cube. Or the planar linkage in Figure \ref{fig:planar-linkage} is extruded to the polyhedral linkage in Figure \ref{fig:extrude}.} to a polyhedral linkage in $3$-dimensions, the domain of the corresponding vertex remains $2$-dimensional. It does not gain an extra dimension of flexibility.

\smallskip

This extruded linkage is shown in Figure \ref{fig:extrude}. The linkage is marked by fixing the vertices $A, B$ and $A'$ to be at $(0,0,0), (1,0,0)$ and $(0,0,1)$, respectively. All other vertices are allowed to move freely, and all edges have length $1$.

\smallskip

\begin{figure}
\begin{center}
    \includegraphics[scale = 0.5]{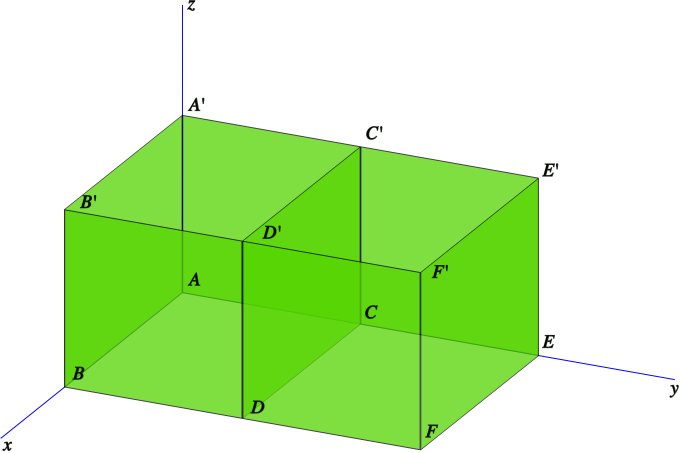} \quad
    \includegraphics[scale = 0.5]{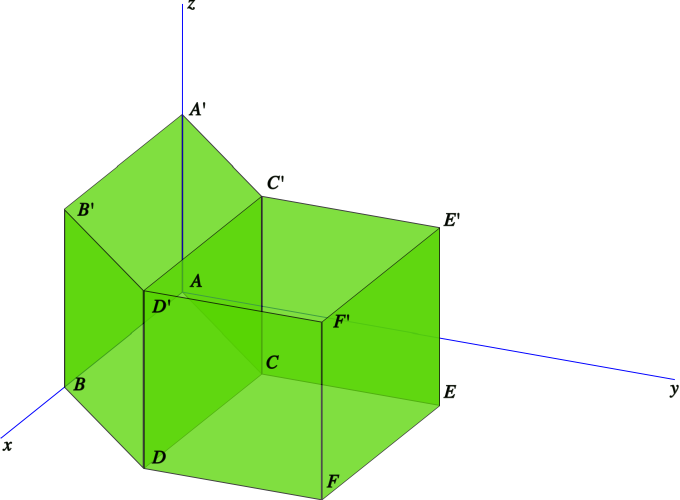}
    \caption{A polyhedral linkage which achieves a $2$-dimensional range of motion.}
    \label{fig:extrude}
\end{center}
\end{figure}

The vertex $E$ may move within the region $\{x^2 + y^2 \leq 4 ; z = 0\}$, which is still $2$-dimensional. We constrain this linkage to restrict the motion to be $1$-dimensional. We add two vertices, $X$ and $Y$, which form two squares $ABYX$ and $EFYX$. See Figure~\ref{fig:roofs}.

\smallskip

\begin{figure}
\begin{center}
    \includegraphics[scale = 0.5]{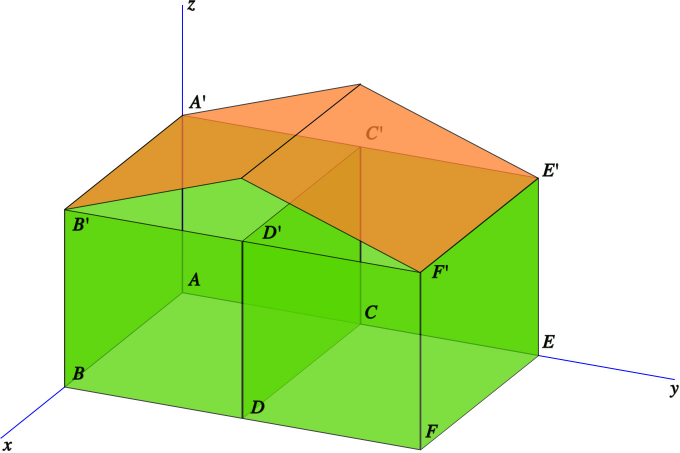} \quad
    \includegraphics[scale = 0.5]{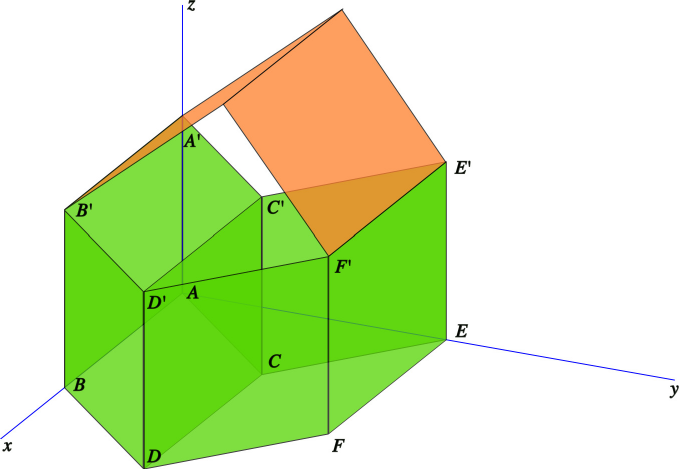}
    \caption{A polyhedral linkage which achieves a $1$-dimensional range of motion.}
    \label{fig:roofs}
\end{center}
\end{figure}

The lengths of all the new edges are set to $1$. The squares $ABYX$ and $EFYX$ constrain the vertices $E$ and $A$ to have the same $x$-coordinate. Therefore, the vertex $E$ may move in the $1$-dimensional region $\{0 \leq y \leq 2 ; x = 0 ; z = 0\}$. This is a notably simpler solution to linear motion in $3$-dimensions than in $2$-dimensions. Note that Sarrus' original linkage is obtained by removing the vertices $D$ and $D'$, and all edges and faces adjacent to them. (See \cite{Sar}.)

\subsection{Periodic construction}
We extend this construction periodically as follows. Note that in any realization of the previous linkage, the squares $EFF'E'$ and $ABB'A'$ are parallel translates of each other by a scalar multiple of the normal vector to each face. We attach multiple copies of the linkage together by identifying the $EFF'E'$ square on the $i$-th copy with the $ABB'A'$ square of the $(i+1)$-st copy.

\smallskip

With $n$ copies attached, the $E$-vertex of the $n$-th copy can be moved within the region\newline$\{0 \leq x \leq 2n ; y = 0 ; z = 0\}$. Moreover, note that if we also add `roofs' connecting the $CD$ edge of the $i$-th copy to the $CD$ edge of the $(i+1)$-st copy, then the $y$-coordinates of the $C$ and $D$ vertices of all constituent linkages will be the same, so they all will flex at the same rate. See Figure~\ref{fig:periodic}.

\smallskip

\begin{figure}
\begin{center}
    \includegraphics[scale = 0.5]{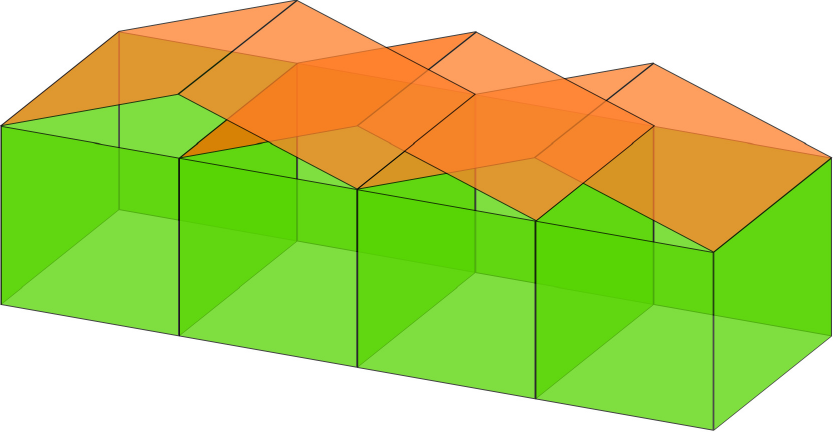} \quad
    \includegraphics[scale = 0.5]{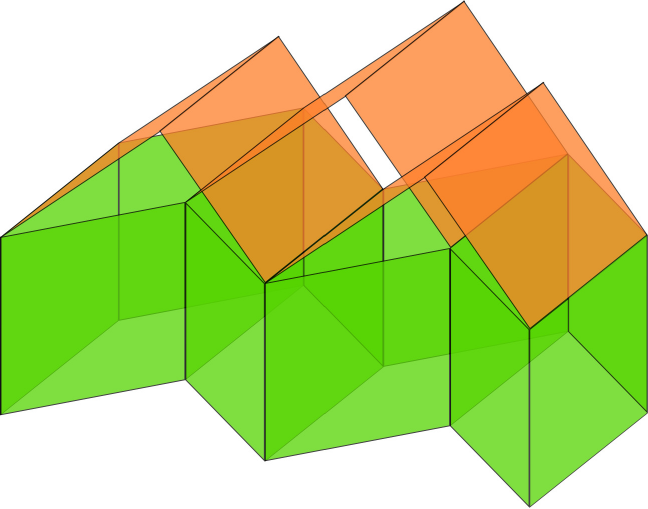}
    \caption{A portion of a periodic polyhedral linkage.}
    \label{fig:periodic}
\end{center}
\end{figure}

Extending infinitely, we obtain a periodic polyhedral linkage whose configuration space is $1$-dimensional, and whose periodicity is preserved during flexion. The direction of the periodicity vector does not change, only the magnitude. This construction can also be embedded by using parallelograms instead of rectangles for the `roofs', see Figure~\ref{fig:embedded-periodic}.

\smallskip

\begin{figure}
\begin{center}
    \includegraphics[scale = 0.5]{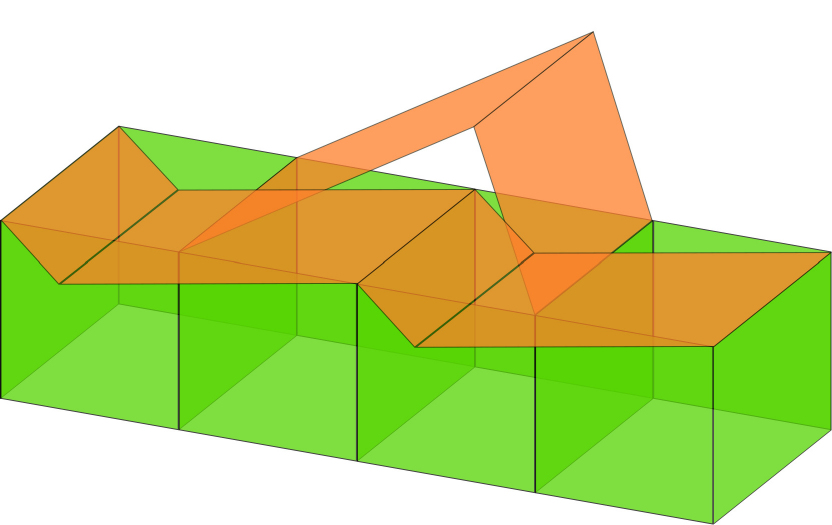} \quad
    \includegraphics[scale = 0.5]{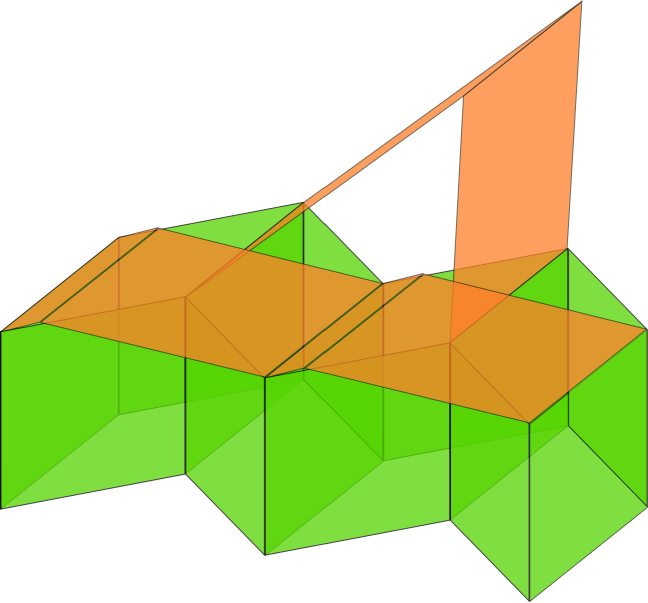}
    \caption{A portion of an embedded periodic polyhedral linkage.}
    \label{fig:embedded-periodic}
\end{center}
\end{figure}

Restricting this linkage to a connected component of the embedded realization space $C^e(\calP, Z)$, we see that each periodic unit has a length between $(0, 2)$ as measured along the $y$-axis. Our construction gives the following theorem:

\begin{theorem}
    \label{thm:extender}
    Let $0 < a < b$ and $\epsilon > 0$. There exists an embedded polyhedral linkage $\calP$ with two faces $F_1, F_2$ and a $1$-dimensional flex such that $F_1$ and $F_2$ remain parallel translates of each other during flexion while the distance between them varies continuously in the range $(a, b)$, and $\calP$ can be contained in a cylinder of radius $\epsilon$ and length $b - a$ under any realization.
\end{theorem}

\begin{proof}
    We scale our given construction so that the side length of each of the squares is a number $d$ such that (1) under any flexion, each unit of the linkage is contained in a ball of radius $\epsilon$, and (2) $b - a = N \cdot 2d$ for some integer $N$. Thus a linkage with $N$ units, along with a rigid component of length $a$, will vary continuously in length between $a$ and $b$.
\end{proof}

We call this construction an \textit{extender linkage}. To simplify future figures, we represent an extender with a red rectangular prism which is understood to extend and contract in length. In contrast, blue faces are rigid faces and cannot be extended. (See Figure \ref{fig:register}.)
{}
\section{Scalar computation}
\label{section:scalar-computation}

Suppose we have an array of extender linkages, pointed parallel to the $y$-axis in the $xy$-plane, and spaced at equal intervals so that the linkages never intersect under any realization. We fix one end of each extender on the $x$-axis and allow the other end to move freely parallel to the $y$-axis in a set range. Each linkage encodes a single real number represented by its length, and in this section we will describe how to use these extenders to perform computations on these inputs via embedded polyhedral linkages. (See Figure \ref{fig:register}.)

\smallskip

\begin{figure}
\begin{center}
    \includegraphics[scale = 0.5]{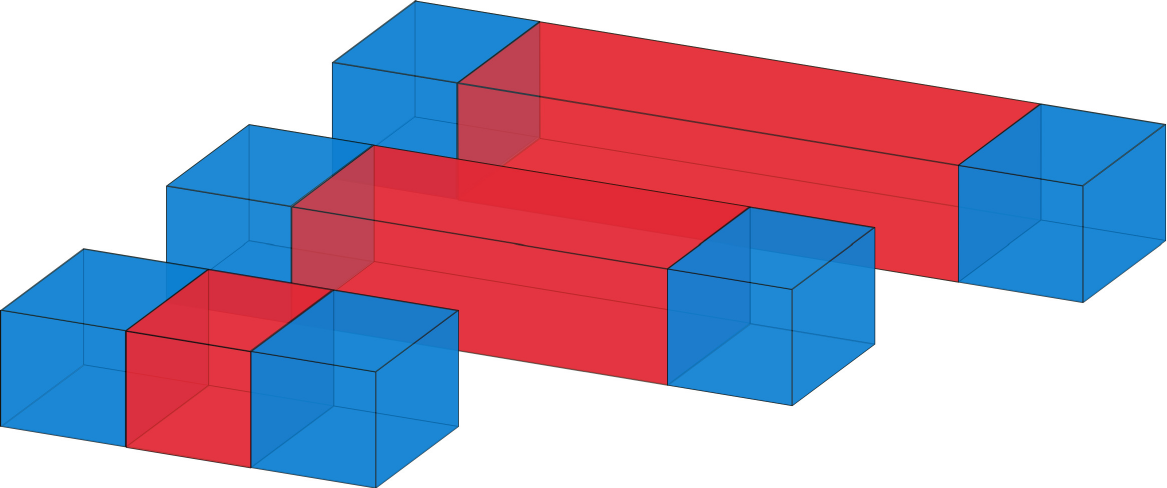}
    \caption{A set of extender linkages representing computational vertices.}
    \label{fig:register}
\end{center}
\end{figure}

Every extender linkage is taken to have the same range, $(a, b)$, for two real numbers $0 << a < b$ to be chosen later. The midpoint $m = \frac{1}{2}(a + b)$ is taken to represent $0$, and in general, an extender linkage that is extended a distance $y$ from the $x$-axis represents the value $y - m$. In this case we call the extender a \textit{computational linkage} and say that $x-m$ is its \textit{value}. We set $N = b - m$, so a computational linkage may represent any value in the range $(-N, N)$. The array of $n$ computational linkages will be called a \textit{register}. We refer to the individual computational linkages as $C_1, \dotsc, C_n$ and their values as $|C_1|, \dotsc, |C_n|$, respectively.

\smallskip

Computations are decomposed into a sequence of elementary unary and binary operations which are performed on registers. Each operation takes one or two computational linkages as inputs, and stores its outputs in the values of another computational linkages. To embed this operation, computations are performed vertically. For a single operation, an output register is created, aligned in the $xy$-plane but offset vertically in the $z$-axis with respect to the input register. The output of the operation is attached to the output computational linkages in the output register. The values of all other original computational linkages are preserved by attaching each to their offset counterpart with rigid linkages.

\smallskip

Let $U \subset (-N, N)$ be a connected set and $f : U \to (-N, N)$ be a function. We say that we can \textit{simulate} the function $f$ by computational linkages if there is an embedded polyhedral linkage which connects an input computational linkage $C_i$ to an output computational linkage $C_i'$, such that if $|C_i| \in U$, then $|C_i'| = f(|C_i|)$. This definition also naturally generalizes if $U$ is a connected subset of $(-N, N) \times (-N, N)$ using two input computational linkages $C_i, C_j$.

\smallskip

In the following subsections, we describe constructions which simulate elementary operations, including addition and multiplication by scalars, negation, addition, and multiplication. Together, these allow us to simulate any polynomial function on any bounded set.

\subsection{Swap and copy}

Before constructing polyhedral linkages to perform actual computations, we need to set up basic operations to manipulate memory stored in registers. We describe operations to swap two values stored in different computational linkages, as well as copy the value of one computational linkage to another.

\smallskip

Let $C_1, \dotsc, C_n$ be the original computational linkages, and let $C_1', \dotsc, C_n'$ be the updated computational linkages. The swap operation $s_i$ sends $|C_i| \mapsto |C_{i+1}'|$ and $|C_{i+1}| \mapsto |C_i'|$. For all $j \not \in \{i, i+1\}$, we have $s_i : |C_j| \mapsto |C_j'|$. By composition, we can freely permute the computational linkages and so we will always assume that the input and output computational linkages are in a desirable configuration.

\smallskip

\begin{figure}
\begin{center}
    \includegraphics[scale = 0.4]{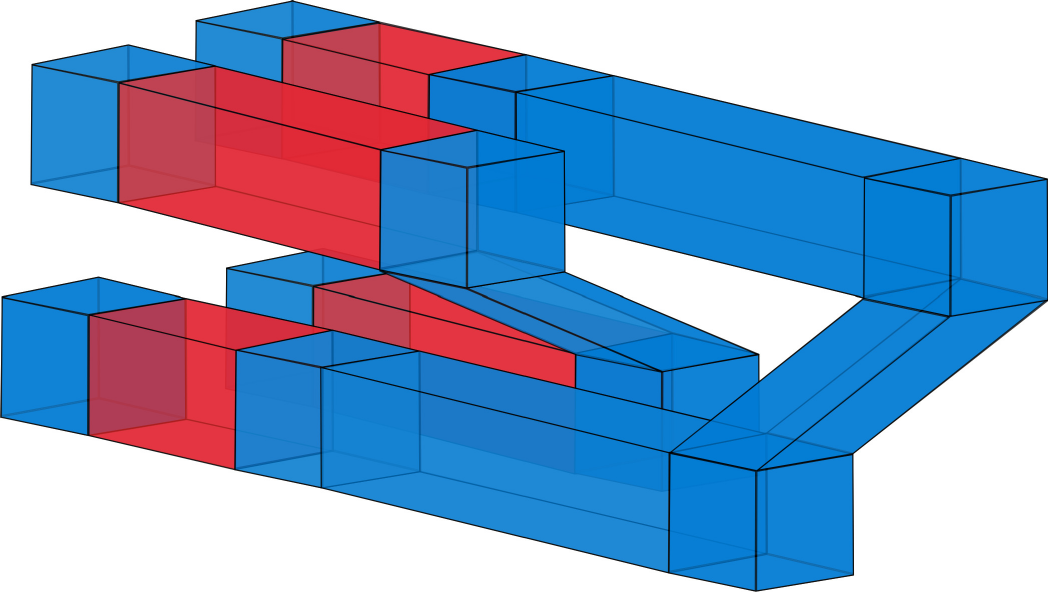} \quad \quad
    \includegraphics[scale = 0.4]{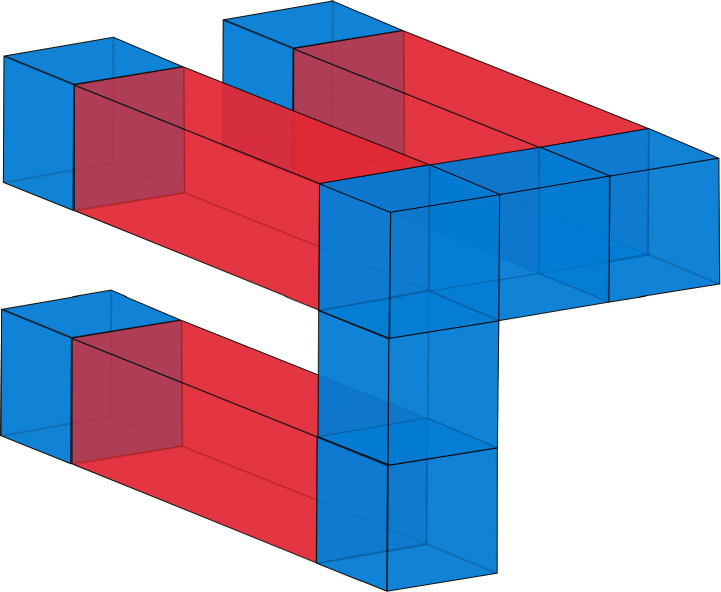}
    \caption{Left: a computational linkage which simulates a swap operation. Right: a computational linkage which simulates a copy operation.}
    \label{fig:swap-copy}
\end{center}
\end{figure}

To simulate $s_i$, first $C_i$ is connected to $C_{i+1}'$ by a rigid set of parallelograms. However, the connection from $C_{i+1}$ to $C_i'$ needs to be rerouted to avoid intersecting with the $C_i$ to $C_{i+1}'$ connection. Thus we add a rectangular prism of length at least $2N$ before adding the parallelograms to ensure that, under any realization, the operation is embedded. All other computational linkages are attached vertically by rectangles. (See Figure~\ref{fig:swap-copy}, left.)

\smallskip

A copy operation is comparatively simpler, because there is no need for rerouting. First, $C_i$ is connected to $C_i'$ by a set of rectangles, and then $C_i'$ is connected to $C_{i+1}'$ by a rigid linkage. The entire construction is rigid, and all other computational linkages are attached vertically. (See Figure~\ref{fig:swap-copy}, right.)

\smallskip

Thus we can always assume without loss of generality that our linkages are in a desirable configuration. Usually, this means that input and output computational linkages are aligned vertically, or else are offset by one or two linkages. Additionally, we will assume enough space between linkages so that our constructions will always be embedded.

\subsection{Scalar addition} 

For any constant $\lambda$, we can simulate the operation $x \mapsto x + \lambda$ on an appropriate domain by adding a rigid linkage.

\begin{lemma}
    For $\lambda \in (0, N)$, the function $x \mapsto x + \lambda$ on the domain $(-N, N - \lambda)$ and the function $x \mapsto x - \lambda$ on the domain $(\lambda - N, N)$ can be simulated by computational linkages.
\end{lemma}

\begin{proof}
    For $\lambda \in (0, N)$, we can define the operation of scalar addition via the linkage shown in Figure~\ref{fig:scalar-addition-negation}, left.) A rigid structure is attached to the end of the computational linkage $C_i$, which adds an offset of length $\lambda$ before attaching to the end of the new computational linkage $C_i'$.

    \begin{figure}
    \begin{center}
        \includegraphics[scale = 0.35]{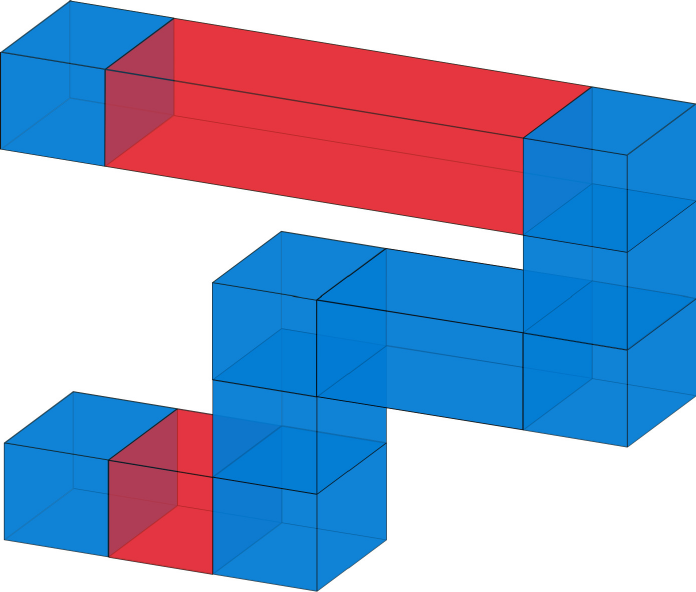} \quad 
        \includegraphics[scale = 0.35]{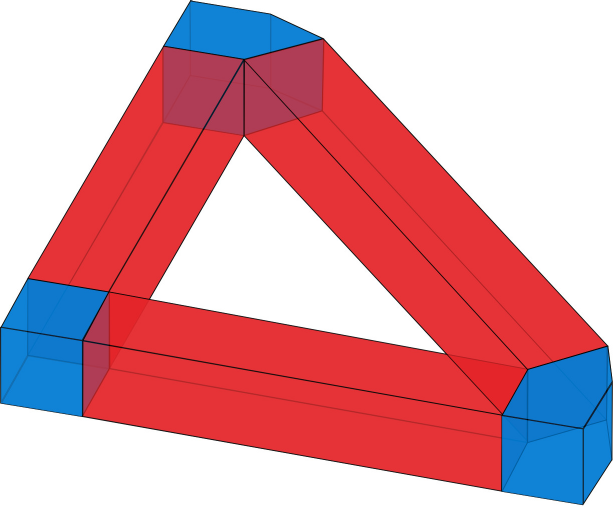} \quad
        \includegraphics[scale = 0.35]{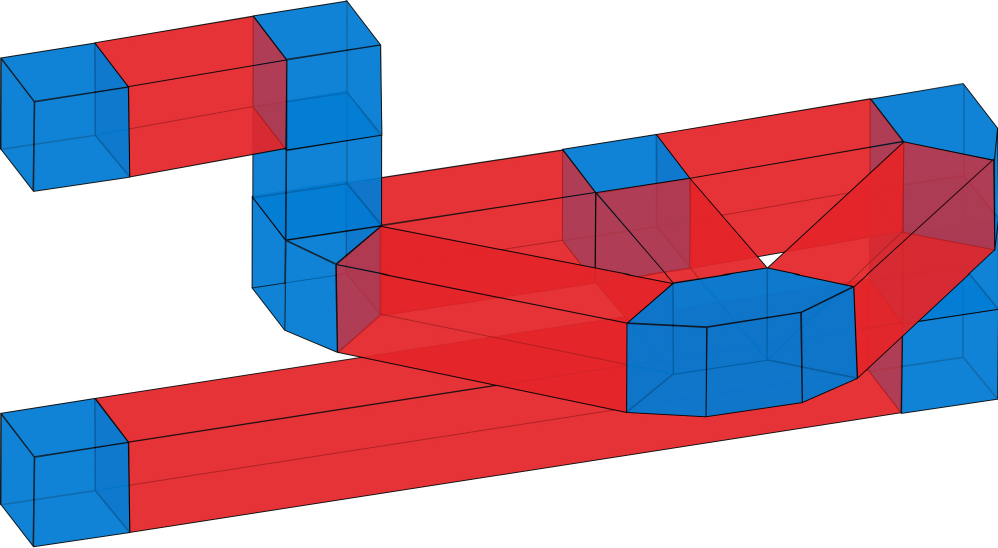}
        \caption{Left: a computational linkage for scalar addition. Middle: a $\left(\frac{\pi}{4}, \frac{\pi}{4}, \frac{\pi}{2}\right)$ triangle linkage. Right: a polyhedral linkage for negation.}
        \label{fig:scalar-addition-negation}
    \end{center}
    \end{figure}
\end{proof}

\subsection{Negation}

We can simulate the operation $x \mapsto -x$ on an appropriate domain by combining several extender linkages orthogonally.

\begin{lemma}
    The function $x \mapsto -x$ on the domain $(-\frac{1}{2}N, \frac{1}{2}N)$ can be simulated by computational linkages.
\end{lemma}

\begin{proof}
    We can simulate the operation of negation $x \mapsto -x$ on the domain $(0, N)$, by using two copies of the $\left(\frac{\pi}{4}, \frac{\pi}{4}, \frac{\pi}{2}\right)$ triangle linkage. (Figure~\ref{fig:scalar-addition-negation}, middle.) Extender linkages are attached at angle $\frac{\pi}{2}$ and connected by a third extended linkage, attached at angle $\frac{\pi}{4}$ relative to both legs. This constrains all three linkages to flex at the same rate, and the perpendicular sides will always be the same length. The negation linkage is formed by joining two $\left(\frac{\pi}{4}, \frac{\pi}{4}, \frac{\pi}{2}\right)$ triangle linkages along a common edge, enforcing the positive and negative legs to have the same length during all realizations. (See Figure~\ref{fig:scalar-addition-negation}, right.) The offset induced by the rigid central cube should be considered, but we can postcompose with a scalar addition operation. We leave the details to the reader.

    \smallskip

    Further, we can adjust the simulated domain to include a neighborhood of $0$ by precomposing and postcomposing with scalar addition $x \mapsto x + \mu$. The largest domain centered at $0$ is obtained by choosing $\mu = N/2$, which lets us define a negation map on the domain $(-\frac{1}{2}N, \frac{1}{2}N)$.
\end{proof}

\subsection{Scalar multiplication}

For any constant $\lambda$, we can simulate the operation $x \mapsto \lambda \cdot x$ on an appropriate domain by using a modified pantograph.

\begin{lemma}
    For $\lambda \in (1, N)$, the function $x \mapsto \lambda x$ on the domain $(-\frac{1}{2\lambda}N, \frac{1}{2\lambda}N)$ can be simulated by computational linkages. For $\lambda \in (0,1)$, the functions $x \mapsto \lambda x$ on the domain $(-\frac{1}{2}N, \frac{1}{2}N)$ can be simulated by computational linkages.
\end{lemma}

\begin{proof}
    We prove the case when $\lambda \in (1, N)$ first. Consider the rigidified pantograph linkage used for scalar multiplication (Figure~\ref{fig:scalar-multiplication-addition}, left) in planar linkages. We fix the lengths $|AC| = \lambda |AB|$ and $|CF| =~\lambda |EF|$. Vertex $A$ is fixed at the origin, $D$ is used as an input vertex, and $F$ is used as an output vertex. The structure is rigidified so that the pairs of edges $AB, BC$ and $CE, EF$ remain parallel. To preserve embeddedness, this is achieved by adding the auxiliary vertices $X$ and $Y$ along with the triangles $ABX, BCX, CEY$ and $EFY$. From the geometry of the construction, the linkage constrains $|AF| = \lambda |AD|$, achieving scalar multiplication. (See \cite[\S 6.2]{KM})

    \smallskip

    \begin{figure}
    \begin{center}
        \includegraphics[scale = 1]{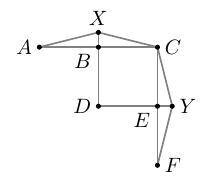}
        \includegraphics[scale = 0.33]{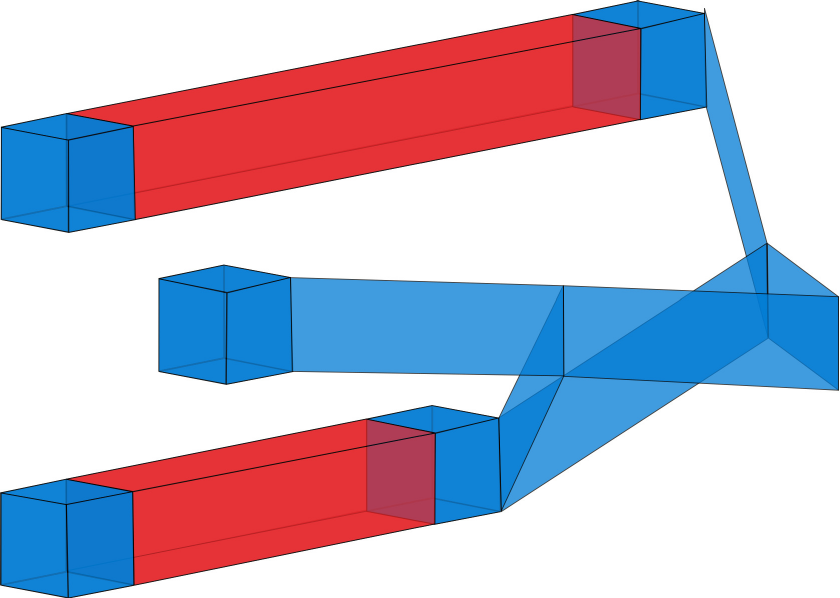} \quad
        \includegraphics[scale = 0.35]{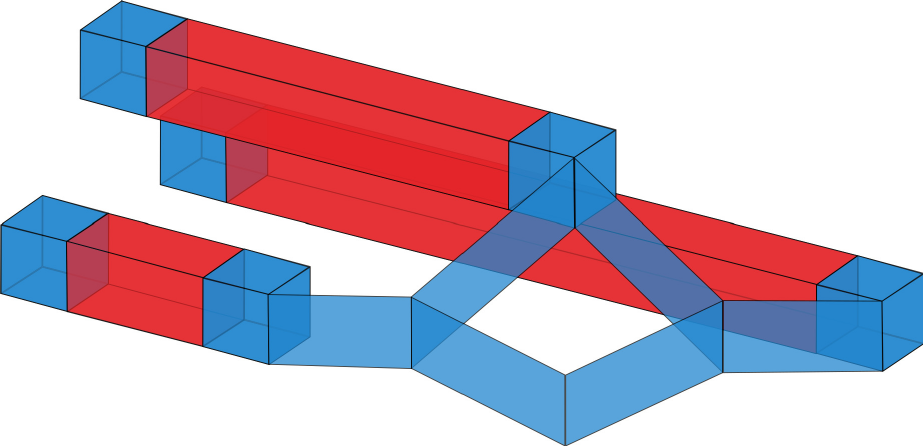}
        \caption{Left: a rigidifed pantograph. Middle: a computational linkage for scalar multiplication. Right: a computational linkage for half-addition.}
        \label{fig:scalar-multiplication-addition}
    \end{center}
    \end{figure}

    We take this construction and extrude it to three dimensions (Figure \ref{fig:scalar-multiplication-addition}, middle.) The vertices corresponding to $D$ and $F$ are offset vertically so that the corresponding computational linkages do not intersect. (Note: for figure clarity, we do not include the extrusion of the auxiliary vertices corresponding to $X$ and $Y$, and the details of rigidifying are left to the reader.) This simulates the operation $x \mapsto \lambda x$ on the domain $(0, \frac{N}{\lambda})$. This domain is limited because nonpositive inputs would cause self intersections.

    \smallskip

    As with negation, we can adjust the simulated domain to include a neighborhood of $0$ by precomposing with $x \mapsto x + \mu$ and then postcomposing with $x \mapsto x - \lambda \mu$. The largest domain centered at $0$ is obtained by choosing $\mu = \frac{N}{2\lambda}$, which lets us define scalar multiplication on the domain $(-\frac{N}{2\lambda}, \frac{N}{2\lambda})$. Lastly, we can handle multiplication by $\lambda \in (0, 1)$ by switching the input and output computational linkages. This is the same process as in planar linkages. (See \cite[\S 6.2]{KM}.) Also note that in this case we can initially simulate the operation $x \mapsto \lambda x$ on the domain $(0, N)$ instead of $(0, \frac{N}{\lambda})$, which leads to the optimal value $\mu = \frac{N}{2}$.
\end{proof}

\subsection{Addition}

We can simulate the operation $(x, y) \mapsto x + y$ on an appropriate domain by using a modified pantograph to perform half-addition and then postcomposing with multiplication by $2$.

\begin{lemma}
    The function $(x, y) \mapsto x + y$ on the domain $(-\frac{1}{4}N, \frac{1}{4}N) \times (-\frac{1}{4}N, \frac{1}{4}N)$ can be simulated by computational linkages.
\end{lemma}

\begin{proof}
    For planar linkages, we can use the pantograph to simulate the function $x, y \mapsto \frac{1}{2}(x + y)$. (Figure~\ref{fig:scalar-multiplication-addition}, left.) Here we fix $\lambda = 2$ and take the vertices $A$ and $F$ to be inputs and the vertex $D$ to be the output. (See \cite[\S 6.3]{KM}) As with scalar multiplication, we extrude the linkage to three dimensions and offset the input and outputs vertically. (See Figure~\ref{fig:scalar-multiplication-addition}, right.)

    \smallskip

    We can simulate the operation $(x, y) \mapsto x + y$ by postcomposition with $x \mapsto 2x$. Multiplication by $2$ is defined on the domain $D := \{x, y \in (-N, N) \ | \ \frac{1}{2}(x + y) \in (-N/4, N/4)\}$. The largest square domain containing $(0,0)$ in $D$ is $(-N/4, N/4) \times (-N/4, N/4) \subset D$, which gives the simulation of the function $(x, y) \mapsto x + y$ on the domain $(-N/4, N/4) \times (-N/4, N/4)$.
\end{proof}

Remark: This proves the stronger statement that we can simulate $(x, y) \mapsto x + y$ on the domain $D$. We use this fact to simplify the construction of other linkages when we have knowledge about the inputs $x$ and $y$. However, in general, we want a definition of the operation whose domain is not dependent on its inputs, and the choice is inconsequential because we will later choose $N$ large enough to neglect this potential inefficiency.

\subsection{Inversion}

We can simulate the operation $x \mapsto \frac{1}{x}$ on an appropriate domain by using a modified Peaucellier inversor.

\begin{lemma}
    The function $x \mapsto \frac{1}{x}$ on the domain $(\frac{1}{N}, N)$ can be simulated by computational linkages.
\end{lemma}

\begin{proof}
    Consider the Peaucellier inversor, which is used for inversion in planar linkages. (See Figure~\ref{fig:peaucellier-inversion}, left.) The vertex $E$ is fixed at the origin. The vertex $C$ is used as an input vertex and the vertex $A$ is used as an output. From the geometry of the construction, we see that $|AE| = \frac{t^2}{|CE|}$ where $t^2 = |DE|^2 - |AD|^2$. We choose edge lengths such that $t^2 = 1$. Note that the classical Peaucillier inversor is subject to degenerate configurations when the square $ABCD$ collapses. (See Kapovich and Millson's use of a `hook' to prevent this in \cite[\S6.4]{KM}.) In our case, however, restricting to the embedded realizations prevents this collapse, so no special treatment is needed. (See also \S \ref{sub-section:embedded}.)

    \smallskip{}

    \begin{figure}
    \begin{center}
        \includegraphics[scale = 1]{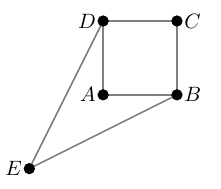}
        \includegraphics[scale = 0.3]{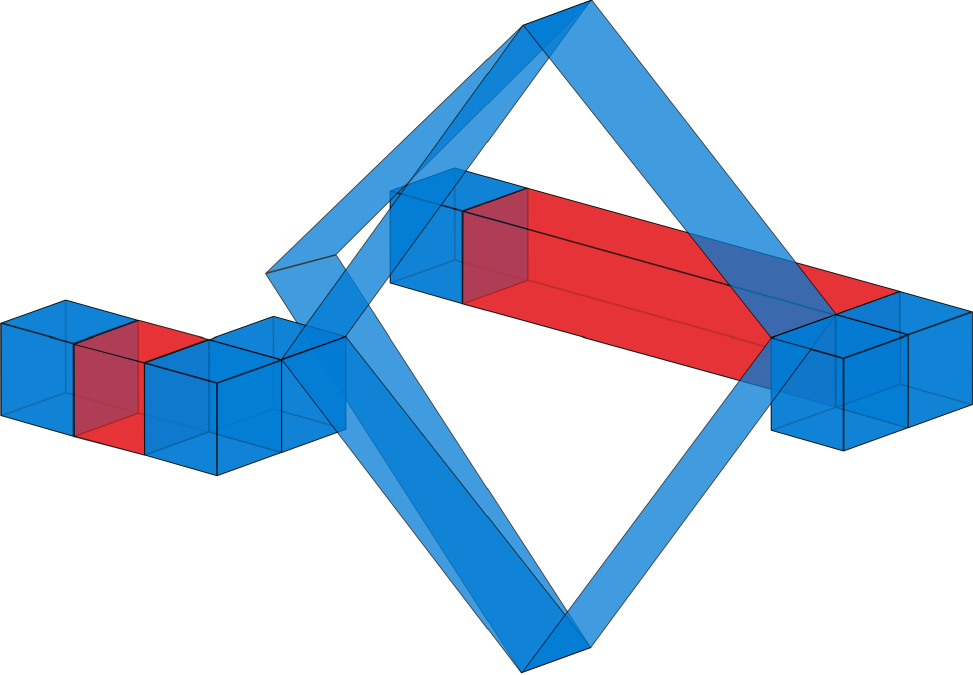}
        \caption{Left: a Peaucellier inversor. Right: a computational linkage for inversion.}
        \label{fig:peaucellier-inversion}
    \end{center}
    \end{figure}

    The edges extruded from the vertices $A$ and $C$ are offset vertically to attach to computational linkages. (See Figure~\ref{fig:peaucellier-inversion}, right.) Choosing appropriate edge lengths, we can simulate the function $x \mapsto \frac{1}{x}$ on the domain $(\frac{1}{N}, N)$, where the bounds come from the possible values representable on computational linkages.
\end{proof}

In this case, we do not worry about finding a domain centered at $0$ because we will never use this linkage directly. We only simulate the operation $x \mapsto \frac{1}{x}$ in order to simulate the operations $x \mapsto x^2$ and $(x, y) \mapsto xy$.

\subsection{Multiplication}

We can simulate the operation $(x, y) \mapsto xy$ on an appropriate domain by using inversion and algebraic identities. To simplify notation, let 
\[
    N^{\ast} = \frac{1}{2} \left(\sqrt{N} - \left(1 + \frac{1}{N}\right)\right)
\] 
for the rest of this section. Importantly, $N^{\ast}$ grows asymptotically as $\sqrt{N}$, so we can always choose a large enough value of $N$ to perform all computations.

\begin{lemma}
    The function $x \mapsto x^2$ on the domain $(-N^{\ast}, N^{\ast})$  can be simulated by computational linkages.
\end{lemma}

\begin{proof}
    We use the identity
    \[
        \frac{1}{2}\left(\frac{1}{x-1} - \frac{1}{x+1}\right) = \frac{1}{x^2 - 1}
    \]
    Thus we can simulate the function $x \mapsto x^2$ by composing with scalar addition, inversion, negation, and addition. The domain is only restricted by the scalar addition and inversion. Note that we only use half-addition,
    \[
        \left( \frac{1}{x-1},\ \frac{-1}{x+1}\right) \mapsto \frac{1}{2} \left(\frac{1}{x-1} + \frac{-1}{x+1}\right)
    \]
    which is defined on the domain $(-N, N) \times (-N, N)$. Also because $x \mapsto \frac{1}{x}$ on the domain $(\frac{1}{N}, N)$ produces only positive outputs, we can use inversion defined on the domain $(0, N)$. Thus no restrictions are imposed by addition and negation.

    \smallskip

    In total, we can simulate the function $\ x ~\mapsto ~x^2$ on the domain $\left(1 + \frac{1}{N}, \sqrt{N}\right)$. To adjust this domain to include $0$, we can precompose by $x \mapsto x + \mu$ and postcompose with $y \mapsto y - 2x \mu - \mu^2$, where $x$ is the original input to the function. The optimal $\mu$ is the average of the endpoints, 
    \[
        \mu = \frac{1}{2} \left(\sqrt{N} + 1 + \frac{1}{N}\right)
    \]
    giving the largest symmetric domain containing $0$ as $(-N^{\ast}, N^{\ast})$. With this value of $\mu$, no additional restrictions come from the pre or postcomposition.
\end{proof}

\begin{lemma}
    The function $(x, y) \mapsto xy$ on the domain $(-\frac{1}{2}N^{\ast}, \frac{1}{2} N^{\ast}) \times (-\frac{1}{2} N^{\ast}, \frac{1}{2}N^{\ast})$ can be simulated on computational linkages.
\end{lemma}

\begin{proof}
    We use the identity
    \[
        \frac{1}{2} \left((x + y)^2 - (x - y)^2 \right) = xy
    \]
    Thus we can simulate the function $x, y \mapsto xy$ by composing addition, negation, and squaring. The composition of adding $x$ and $y$ and then squaring adds a new restriction on the domain, so we require that $x, y \in (-\frac{1}{2} N^{\ast}, \frac{1}{2} N^{\ast})$. The other operations do not impose any new restrictions, where again we note that in the last step we use half-addition, $x, y \mapsto \frac{1}{2}(x + y)$, which is defined everywhere.
\end{proof}

\subsection{All polynomials can be simulated by computational linkages.}

With the above constructions, we can simulate any polynomial function on an appropriate domain.

\begin{theorem}
    Let $f : \R^n \to \R^m$ be a polynomial function and $U \subset \R^n$ be a bounded set. Then we can simulate $(f, U)$ by computational linkages.
\end{theorem}

\begin{proof}
    A polynomial function comes from the composition of scalar addition, scalar multiplication, negation, addition and multiplication. Each operation can be simulated by computational linkages, and for each operation, the domain is defined on some neighborhood of $0$ whose size only depends on $N$. Thus we can choose $N$ large enough that each operation is always defined on every input from $U$.
\end{proof}

\section{Vector computation}
\label{section:vector-computation}

For $3$ dimensional motion, we consider the $1$-skeleton of a cube made of twelve extender linkages joined by eight rigid cubes. (See Figure~\ref{fig:cube}, left.) If one cube is fixed, the opposite cube has a $3$-dimensional range of motion. Moreover, note that if this cube records the position $(x, y, z)$, then the three cubes adjacent to the fixed cube record the coordinates $(x, 0, 0), (0, y, 0)$ and $(0, 0, z)$.

\smallskip

\begin{figure}
\begin{center}
    \includegraphics[scale = 0.3]{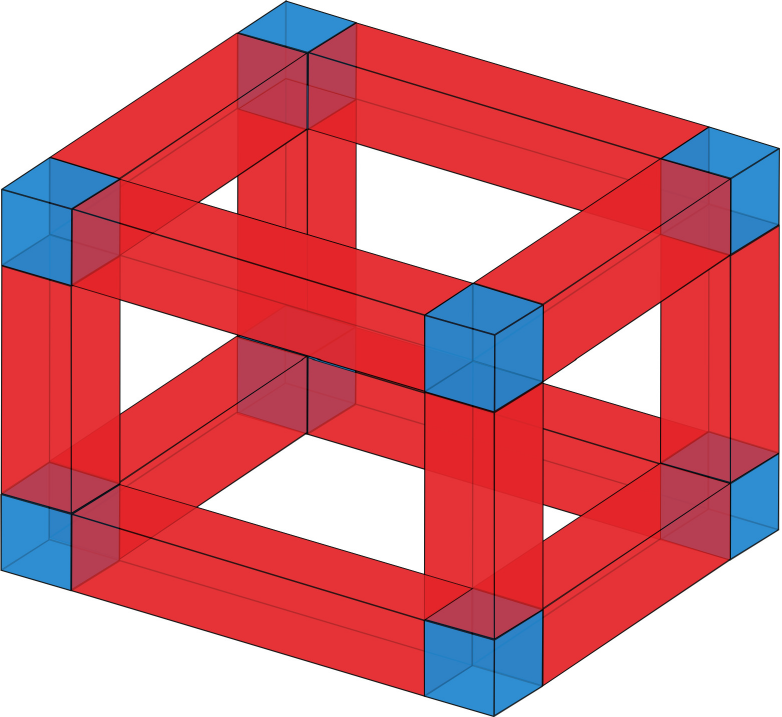} \quad
    \includegraphics[scale = 0.3]{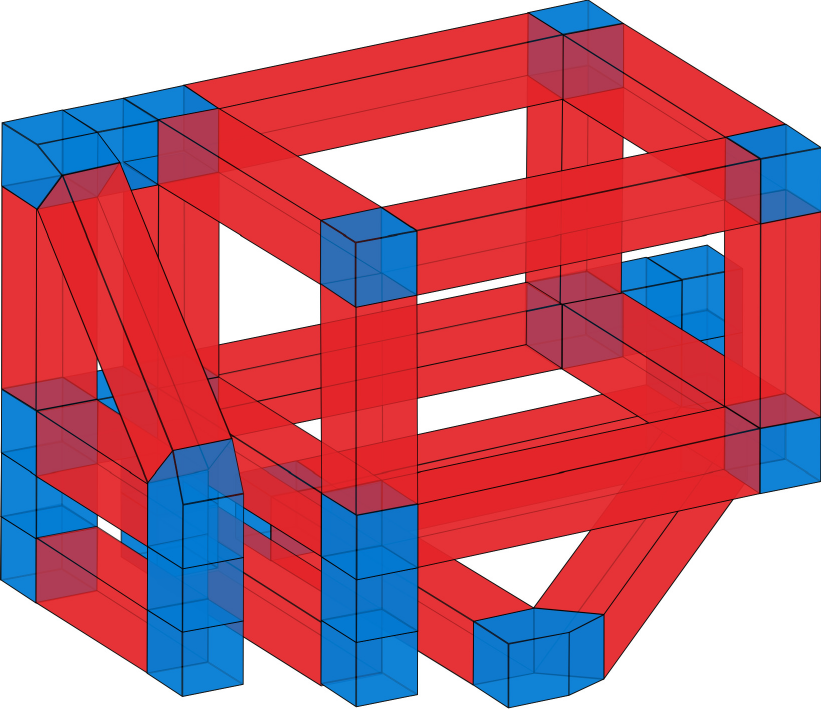} \quad
    \includegraphics[scale = 0.3]{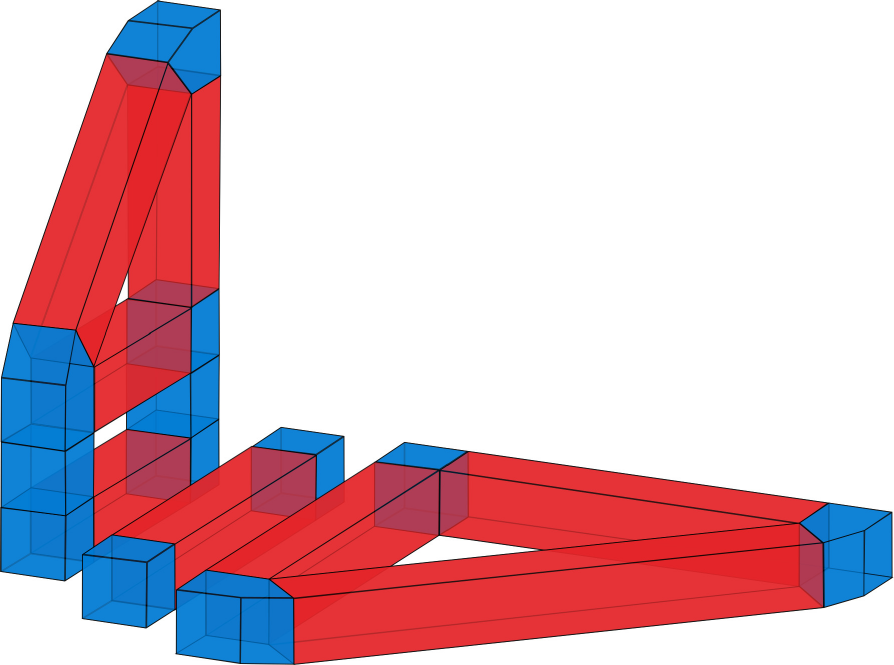}
    \caption{Left: an embedded polyhedral linkage with a $3$-dimensional flexion. Middle/Right: a mechanism for translating $3$ dimensional motion into computational linkages.}
    \label{fig:cube}
\end{center}
\end{figure}

By adding two $\left(\frac{\pi}{4}, \frac{\pi}{4}, \frac{\pi}{2}\right)$ triangle linkage, we can transfer $3$ dimensional motion to an array of three computational linkages. By offsetting the linkages, this construction can be embedded. (See Figure \ref{fig:cube} middle, right.)

\smallskip

Conversely, we can also take three computational linkages $C_1, C_2, C_3$ and simulate three dimensional motion with a vertex aligned at $(|C_1|, |C_2|, |C_3|)$. This proves a generalization to the result given in $1$ dimension on computational linkages.

\begin{customthm}{1.1}
    Let $F : \R^{3m_1} \to \R^{3m_2}$ be a polynomial function, and $U \subset \R^{3m_1}$ be a bounded open set. There exists an embedded polyhedral linkage which realizes $F$ on $U$.
\end{customthm}

As a corollary, we can also restrict a polyhedral linkage to trace out a semialgebraic set.

\begin{customcor}{1.2}
    Let $V \subset \R^3$ be an algebraic set, and $U \subset \R^3$ be a bounded subset. There there is a fixed embedded polyhedral linkage which realizes $V \cap U$.
\end{customcor}

\begin{proof}
    An algebraic set $V$ is defined as the vanishing locus of a set of polynomials $f_1, \dotsc, f_n : \R^3 \to \R$. Let $F$ be the function $\R^3 \to \R^{3n}$ defined by
    \[
        F : (x, y, z) \mapsto \big((f_1(x,y,z), 0, 0), \dotsc, (f_n(x,y,z), 0, 0) \big)
    \]
    In this linkage we fix all of the output vertices to be at their relative origins $(0,0,0)$, thus for each $f_i$, we have the constraint that $f_i(x, y,z) = 0$. Thus the input vertex is constrained to be in the intersection $V \cap U$, completing the proof.
\end{proof}

\section{Final remarks}

\subsection{Higher dimensions}
\label{sub-section:higher-dim}
Note that our proofs for Theorem \ref{thm:linkage-universal} and Corollary \ref{thm:linkage-universal-closed} naturally generalize to all dimensions $n > 3$. By extruding a polyhedral linkage into higher dimensions, no added flexibility is gained. Thus by extruding our construction of extender linkages, we can generalize the results of Theorem~\ref{thm:extender} into higher dimensions. Each step for performing scalar computation will carry over as well. And finally, we can modify our construction of translating $3$-dimensional motion into computational linkages by employing the $1$-skeleton of an $n$-cube made of extender linkages to translate $n$-dimensional motion into a register of $n$ extender linkages, as in Figure~\ref{fig:cube}. Thus our results hold for embedded polyhedral linkages in dimension $n \geq 3$.

\subsection{Embedding planar linkages}
\label{sub-section:inversion}
The Peaucellier inversor is the main barrier to embedding planar linkages in Kempe's original proof. Although the Peaucellier inversor itself is planar, because the output vertex $A$ is not adjacent to the external face, after any composition the resulting planar linkage will not correspond to a planar graph. (See Figure \ref{fig:peaucellier-inversion}.) Thus any planar linkage which represents a polynomial of degree $\geq 2$ cannot be embedded in $\R^2$ via Kempe's construction. An embedded construction was recently achieved in \cite{AB+}.

\smallskip

However, note that deciding if a planar linkage can be embedded is different than deciding if the underlying graph is planar. F\'{a}ry's Theorem \cite{Fary} states that every planar graph can be drawn in the plane using straight lines, so every planar graph can be realized by a linkage for some choice of edge weights. However, there exist planar linkages embedded in $\R^3$ which are topologically equivalent to the unknot but which cannot be moved to a convex realization. Such linkages are called \textit{locked}. (See e.g. \cite{CJ} and \cite{BD+}.) This phenomena is special to $3$-dimensions; there are no locked planar linkages in $2$-dimensions or $4$-dimensions. (See \cite{CDR} and \cite{CO'R}.)

\smallskip

With suitable hypotheses, Wilson \cite{Wil} proved a generalization of F\'{a}ry's Theorem for topological simplicial complexes in $\R^n$. Because polyhedral linkages have less flexibility than planar linkages, the existence of locked planar linkages in $\R^3$ implies that the same is true for $2$-dimensional polyhedral linkages.

\subsection{Embedded realizations}
\label{sub-section:embedded}
In $\S$ \ref{sub-section:polyhedral-linkages}, we chose the convention to only consider the subset of embedded realizations so that the resulting linkages could be realized as physical mechanisms. An alternate approach is to constrain the linkages so that only embedded realizations are possible. (See e.g. \cite{AB+} for similar considerations for planar linkages in $\R^2$.)

\smallskip

\begin{figure}
\begin{center}
    \includegraphics[scale = 1]{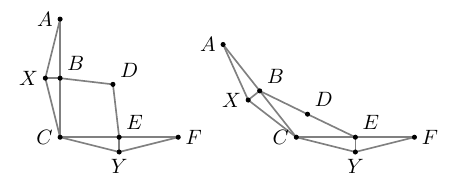}
    \caption{A skew pantograph that achieves a bounded flexible angle.}
    \label{fig:dihedral}
\end{center}
\end{figure}

Consider a rigidified pantograph, with $|BD| = |DE| < |BC| = |CE|$. (See Figure \ref{fig:dihedral}.) If $t = \frac{|BD|}{|BC|}$, then the angle $\angle ACF$ can open to a maximum measure of $2 \sin^{-1}(t) < \pi$. Thus another way to prevent degenerate realizations of the square planar linkage is to replace each corner of the square with a copy of the skew pantograph. Then no degenerate realization is possible because every angle of the original square is strictly less than $\pi$. Extruding to a $3$-dimensional polyhedral linkage gives a linkage whose entire configuration space is embedded. So one could provide an alternate proof of Theorem \ref{thm:extender} based on this construction and produce an extender linkage whose entire realization space is embedded.

\subsection*{Acknowledgements}
I am grateful to Igor Pak for his advice on how to structure this paper, and to Alexey Glazyrin for feedback on early drafts. I am also grateful to Joseph O'Rourke, Simon Guest and Zeyuan He for feedback and direction.

\end{document}